\documentclass[12pt,leqno,fleqn]{amsart}
\usepackage{amsmath,amstext,amsthm,amssymb,amsxtra}
\usepackage{txfonts,pxfonts} 
\usepackage[colorlinks,citecolor=red,pagebackref,hypertexnames=false]{hyperref}
\usepackage[backrefs]{amsrefs}

\usepackage{pgf,tikz}
\allowdisplaybreaks
\swapnumbers

\makeatletter \newcommand \listoftodos{\section*{Todo list} \@starttoc{tdo}}
\newcommand\l@todo[2]
    {\par\noindent \textit{#2}, \parbox{10cm}{#1}\par} \makeatother

\theoremstyle{plain} 
\newtheorem{lemma}[equation]{Lemma}

\newtheorem{theorem}[equation]{Theorem}

\newtheorem*{roth}{K.~Roth's Theorem}
\newtheorem*{schmidt}{W.~Schmidt's Theorem}

\theoremstyle{definition}

\theoremstyle{remark}

\numberwithin{equation}{section}

\def\norm#1.#2.{\lVert#1\rVert_{#2}}
\def\Norm#1.#2.{\bigl\lVert#1\bigr\rVert_{#2}}
\def\NOrm#1.#2.{\Bigl\lVert#1\Bigr\rVert_{#2}}
\def\NORm#1.#2.{\biggl\lVert#1\biggr\rVert_{#2}}
\def\NORM#1.#2.{\Biggl\lVert#1\Biggr\rVert_{#2}}

\def\ip#1,#2,{\langle #1,#2\rangle}
\def\Ip#1,#2,{\bigl\langle#1,#2\bigr\rangle}
\def\IP#1,#2,{\Bigl\langle#1,#2\Bigr\rangle}

\def\Abs#1{\bigl\lvert#1\bigr\rvert}

\def\ABS#1{\Bigl\lvert#1\Bigr\rvert}

\def\XXint#1#2#3{{\setbox0=\hbox{$#1{#2#3}{\int}$}
     \vcenter{\hbox{$#2#3$}}\kern-.5\wd0}}

\def\eqdef{\stackrel{\mathrm{def}}{{}={}}}

\begin{document}
\title[Cyclic shifts of the Van der Corput set]{Cyclic shifts of the Van der Corput set}

\thanks{The author is  grateful to the Fields Institute and the Institute for Advanced Study for hospitality
 and to the National Science Foundation for support.}

\author[Dmitriy Bilyk]{Dmitriy Bilyk }
\address{University of South Carolina, Columbia, SC /
Institute for Advanced Study, Princeton, NJ}
\email{bilyk@math.ias.edu}

\begin{abstract}
In \cite{roth4}, K. Roth showed that the expected value of the $L^2$
discrepancy of the cyclic shifts of the $N$ point van der Corput set
is bounded by a constant multiple of $\sqrt{\log N}$, thus
guaranteeing the existence of  a shift with asymptotically minimal
$L^2$ discrepancy, \cite{MR0066435}. In the present paper, we
construct a specific example of such a shift.
\end{abstract}

\maketitle

\section{Introduction} 

Let $ \mathcal A_N \subset [0,1] ^{2}$ be a finite point set of
cardinality $ N$.  The extent of equidistribution of $\mathcal A_N$
can be measured by the \emph{discrepancy function}:
\begin{equation*}
D_{\mathcal A_N} (x_1,x_2) \coloneqq \sharp \bigl(\mathcal A_N \cap
[0, x_1)\times[0,x_2) \bigr) - N x_1\cdot x_2\,,
\end{equation*}
i.e. the difference between the actual and expected number of points
of $\mathcal A_N$ in the rectangle $[0, x_1)\times[0,x_2)$. The main
principle of the theory of {\em irregularities of distribution}
states that the size of this function must increase with $ N$. The
fundamental results in the subject are:

\begin{roth}(\cite{MR0066435}, 1954) For any set $ \mathcal A_N \subset
[0,1]^{2}$, we have
\begin{equation}\label{e.roth}
\norm D_{\mathcal A_N}. 2. \gtrsim (\log N) ^{1/2}
\end{equation}
where ``$\gtrsim$" stands for ``greater than a constant multiple
of".
\end{roth}



\begin{schmidt} (\cite{MR0319933}, 1972)
For any set $ \mathcal A_N \subset [0,1]^2$, we have
\begin{equation}\label{e.schmidt}
\norm D_{\mathcal A_N} .\infty . \gtrsim \log N \,.
\end{equation}
\end{schmidt}
Both theorems are known to be sharp in the order of magnitude (e.g.,
\cite{Cor35}, \cite{MR0082531}, \cite{roth3}, \cite{MR610701}). One
of the most famous examples, yielding sharpness of
\eqref{e.schmidt}, is the van der Corput ``digit-reversing" set,
\cite{Cor35}. For $N=2^n$ points, it can be defined as
\begin{equation}\label{e.vdc}
\mathcal V_n = \left\{ (0.a_1 a_2 \dots a_n 1,\, 0.a_n a_{n-1} \dots
a_2 a_1 1 ): \, a_i=0,1 \right\},
\end{equation}
where the coordinates are given in terms of the binary expansion.
Unfortunately, most ``classical" sets with minimal $L^\infty$ norm
of the discrepancy fail to meet the sharp bounds in the $L^2$ norm.
In fact, Halton and Zaremba \cite{HZ} proved that
\begin{equation}\label{e.hz}
\|D_{\mathcal V_n}\|_2^2 = \frac{n^2}{2^6}+ O(n)\approx (\log N)^2.
\end{equation}

There are three standard remedies in the theory for this
shortcoming. To achieve the  smallest possible order of the $L^2$
discrepancy, one can alter the sets in the following ways:

{\em 1. Davenport's Reflection Principle.} Informally, if $P$ has
low $L^\infty$ discrepancy, then the set $\widetilde{P} = P \cup
\{(1-x,y):\, (x,y) \in P \}$ has low $L^2$ discrepancy. This was
demonstrated by Davenport \cite{MR0082531} in the case of the
irrational lattice, and by Chen and Skriganov (\cite{MR1805869}, see
also \cite{MR0965955}) in the case of the van der Corput set.

{\em 2. Digit Scrambling.} This procedure, initially introduced in
\cite{chen2}, has been extensively studied; a comprehensive
discussion can be found in \cite{MR1697825}.

{\em 3. Cyclic shifts.} This transformation is the subject of this
paper. It has been proved by Roth, \cite{roth4}, (see also
\cite{roth3}, where the translation idea was originally used), that
for the cyclic shifts of the van der Corput set
\begin{equation}
\mathcal V^\alpha_n = \left\{ \big( (x+\alpha)\, \textup{mod}\, 1, y
\big):\, (x,y)\in \mathcal V_n \right\},
\end{equation}
the expected value of the $L^2$ discrepancy over $\alpha$ satisfies
\begin{equation}
\int_0^1 \| D_{\mathcal V_n^\alpha}\|_2 \, d\alpha \lesssim n^{1/2}
= \left( \log N \right)^{1/2}.
\end{equation}
This implies that there exists a particular cyclic shift of the van
der Corput distribution with minimal $L^2$ norm of the discrepancy
function. However, this was purely an existence proof and no
deterministic examples of such shifts have been constructed. In the
present paper, we ``de-randomize" this result and provide an
explicit value of $\alpha$, which asymptotically minimizes
$\|D_{\mathcal V_n^\alpha}\|_2$. We prove the following theorem
\begin{theorem}\label{t.main}
For $\alpha_0=1-\frac{k}{2^n}$, where $k\in \mathbb N$, in the
binary form, is given by
\begin{equation}\label{e.defk}
k := \big( \underbrace{000 \,\dots \, 00}_{n_0 \, \textup{digits}}
\underbrace{00001111 \,\dots \, 00001111}_{n_2 \, \textup{digits}}\,
\underbrace{000111 \,\dots \, 000111}_{n_1 \, \textup{digits}}
\big)_2 + 1,
\end{equation}
with $n_0+n_1+n_2=n$, $\frac{n_1}{n_2}=\frac{54}{17}$, and
$n_0<568$, the cyclically shifted van der Corput set $\mathcal
V_n^{\alpha_0}$ satisfies
\begin{equation}\label{e.main}
\Norm D_{\mathcal V_n^{\alpha_0}}. 2. \lesssim n^{1/2} = (\log
N)^{1/2}.
\end{equation}
\end{theorem}

{\em Remark.} The ``$+1$" in the end of \eqref{e.defk} is just a
minor nuisance, which simplifies some calculations, and is not
important. In fact, one can easily see that a cyclic shift by the
amount $\alpha = 1/N = 2^{-n}$ changes the discrepancy by at most
$1$ at each point.

We would like to point out that most constructions of sets with
minimal order of $L^p$ discrepancy (which are important in
applications to numerical integration) are probabilistic; explicit
constructions are rare. In fact, the first deterministic examples of
such sets in dimensions $d\ge 3$ have only been obtained quite
recently by Chen and Skriganov (\cite{MR1896098}, \cite{MR2283797}).

 The outline of the paper is the following: in \S2 we deal with the
quantities $\int_{[0,1]^2} D_{\mathcal V_n}(x)dx$ and
$\int_{[0,1]^2} D_{\mathcal V_n^\alpha}(x)dx$ (which can be viewed
as the ``zero-order" term of the expansion in any reasonable
orthonormal basis) and minimize the latter. In \S 3, we examine the
Fourier coefficients $\widehat{D_{\mathcal V_n}}(n_1, n_2)$ when
$(n_1,n_2)\neq (0,0)$  and show that they do not change too much
under cyclic shifts.

We will refer to the two parts of the discrepancy function as
``linear" and  ``counting":
\begin{align}\label{e.linear}
L_N (x_1,x_2) &= N x_1 \cdot x_2 \,,
\\  \label{e.counting}
C_ {\mathcal A_N} (x_1,x_2) &= \sum _{ p \in \mathcal A_N} \mathbf
1_{[p_1, 1)\times [p_2,1) } ( x_1,x_2) \,.
\end{align}
In proving upper bounds for the discrepancy function, one of course
needs to capture a large cancelation between these two.

\section{The integral of the discrepancy function}
Recall that in our definition of the van der Corput set, $\mathcal
V_n = \left\{ (0.a_1  \dots a_n 1,\, 0.a_n  \dots a_2 a_1 1
)\right\}$, both coordinates have $1$'s in the $(n+1)^{st}$ binary
place. This is just a technical modification, which ensures that,
for any $\alpha =j/2^n$, $j\in \mathbb Z$, the average value of both
coordinates in $\mathcal V_n^\alpha$ is one-half:
\begin{equation}\label{e.sumall}
\frac1{2^n}\sum_{(p_1,p_2)\in \mathcal V_n^\alpha} p_1=
\frac1{2^n}\sum_{(p_1,p_2)\in \mathcal V_n^\alpha} p_2 =\frac12.
\end{equation}
This makes many formulas look `cleaner' and is not essential to the
computations.

It has been noticed (see \cite{HZ}, \cite{BLPV}), that the quantity
$\int_{[0,1]^2} D_{\mathcal V_n}(x)dx$ is the main reason why $\|
D_{\mathcal V_n} \|_2$ is large. Indeed, if one compares
\eqref{e.hz} and \eqref{e.integral} below, it is easy to see that
\begin{equation}\label{e.rest}
\NOrm D_{\mathcal V_n} - \int D_{\mathcal V_n} .2. \lesssim (\log
N)^{1/2}.
\end{equation}
We include the proof of the lemma below for the sake of
completeness.
\begin{lemma}\label{l.integral} For the  van der Corput set $ \mathcal V_n$
\begin{equation}\label{e.integral}
\int _{[0,1) ^2 } D_{\mathcal V_n} (x) \; d x = \frac n8.
\end{equation}
\end{lemma}

\begin{proof} The linear part of the discrepancy function clearly gives us
\begin{equation}\label{e.I1}
\int _{[0,1) ^2 } L_N  \; d  x= 2 ^{n-2}\,.
\end{equation}

Let $ X_1 ,\dotsc, X_n$ be independent random variables taking
values $ \{0,1\}$ with probability $\frac12$. A straightforward
computation yields


\begin{align} \notag
\int _{[0,1) ^2 } C _{\mathcal V_n} (x_1, x_2) \; dx_1\, dx_2 &=
\sum _{(p_1,p_2) \in \mathcal V_n} (1-p_1)(1-p_2)\\
 \notag & = 2
^{n} \mathbb E \Biggl[ 1- \sum _{j=1} ^{n} X_j 2 ^{-j} -
2^{-n-1}\Biggr] \Biggl[ 1-\sum _{k=1} ^{n} X_k 2 ^{-n+k-1} -
2^{-n-1}\Biggr]
\\
\label{e.I4} &= 2 ^{n-2} +\frac{n}{8}.
\end{align}
Combining \eqref{e.I1} and \eqref{e.I4} proves the lemma.

\end{proof}
%
%


In what follows we prove that the average of $\int_{[0,1]^2}
D_{\mathcal V_n^\alpha} dx $ over $\alpha$ is zero. Besides, we
construct a specific value of $\alpha_0$, for which $$
\int_{[0,1]^2} D_{\mathcal V_n^{\,\alpha_0}}\,\, dx \approx 1.$$


\begin{theorem}
 Assume that $\alpha\in [0,1)$ is an $n$-digit binary number. Then
 \begin{equation}
 \mathbb E_\alpha \int_{[0,1]^2} D_{\mathcal V_n^\alpha} dx
 = 0.
 \end{equation}
 \end{theorem}

\begin{proof}
 We denote $1-\alpha =\frac{k}{2^n}$ ($k=1,\dots , 2^n$) and
start
  with the following  computation:

\begin{align}
\nonumber \int_{[0,1]^2} C_{\mathcal V_n^\alpha} & =  \sum_{p\in
\mathcal
V_n^\alpha} (1-p_1)(1-p_2)\\
\nonumber & =  \sum_{p\in \mathcal V_n:\, p_1< 1-\alpha}
(1-p_1-\alpha)\cdot(1-p_2)\,\, + \sum_{p\in \mathcal V_n:\, p_1 >
1-\alpha} (2-p_1-\alpha)\cdot(1-p_2)\\
\nonumber & = \int_{[0,1]^2} C_{\mathcal V_n} dx +
(1-\alpha)\sum_{p\in \mathcal V_n} (1-p_2) - \sum_{p\in \mathcal
V_n:\, p_1< 1-\alpha}
(1-p_2)\\
\label{e.cn}&= \int_{[0,1]^2} C_{\mathcal V_n} dx \,\,-\,\,
\frac{k}{2} + \sum_{p\in \mathcal V_n:\, p_1< k/2^{n}} p_2.
\end{align}
Next, we examine the behavior of the last sum above. Using the
structure of the van der Corput set, we can write
\begin{equation}\label{e.fk}
\sum_{p\in \mathcal V_n:\, p_1< k/2^{n}} p_2 =\sum_{l=1}^n 2^{-l}
f_l (k)+ k 2^{-n-1},
\end{equation}
where $k 2^{-n-1}$ comes from the final $1$'s in the expansion of
$p_2$ and
\begin{equation}\label{e.defflk}
f_l (k)=\#\{0\le j \le k-1:\,\,{\textup{ such that the $l^{th}$
(from the end) 
binary digit of
$j$ is $1$}} \}.
\end{equation}
It can be seen that
\begin{equation}\label{e.flk1}
f_l(k)=2^{l-1}m \,\,\,\,\,\,\,\,\,\,\,\,\, \textup{if } k-1=2^l m,\,
2^l m +1,\, ... \,, 2^l m + 2^{l-1}-1,\,\,\, \textup{and}
\end{equation}
\begin{equation}\label{e.flk2}
f_l(k)=2^{l-1}m + j \,\,\,\,\,\,\,\,\, \textup{if } k-1=2^l
m+2^{l-1}+j-1,\textup{ where }\,1\le j \le 2^{l-1}.
\end{equation}
Thus, if we set $f_l(k) = 2^{l-1} m_l(k) + j_l(k)$, where $0\le
m_l(k) < 2^{n-l}$ and $1\le j_l(k) \le 2^{l-1}$, we have $\mathbb
E_k m_l(k) = \frac12 \cdot (2^{n-l}-1)$ and
$$\mathbb E_k j_l(k)=\frac12 \cdot \frac12(2^{l-1}+1), $$
where the extra one-half above comes from the fact that $j_l(k)=0$
half of the time. Thus
\begin{equation}
\mathbb E_k f_l(k) = 2^{n-2}  - 2^{l-2} + 2^{l-3} +\frac14.
\end{equation}
Plugging this into \eqref{e.fk}, we obtain
\begin{align}
\nonumber \mathbb E_k \sum_{p\in \mathcal V_n:\, p_1< k/2^{n}} p_2
& = \sum_{l=1}^n 2^{-l} \left( 2^{n-2} - 2^{l-3} +\frac14 \right)\, + \mathbb E_k k\cdot 2^{-n-1}\\
\label{e.cn1} & = 2^{n-2} - \frac{n}{8} +\frac14.
\end{align}
Finally, equation \eqref{e.cn}, together with \eqref{e.cn1} as well
as \eqref{e.integral}, yields
\begin{align}
\nonumber \mathbb E \int_{[0,1]^2} D_{\mathcal V_n^\alpha}  &=
\int_{[0,1]^2} D_{\mathcal V_n} dx \,\,-\,\,
\mathbb E_k \frac{k}{2} + \mathbb E_k \sum_{p\in \mathcal V_n:\, p_1< k/2^{n}} p_2\\
& = \frac{n}{8}\, - \,\left(2^{n-2} +\frac14 \right) \,+\, \left(
2^{n-2} - \frac{n}{8} +\frac14 \right) = 0.
\end{align}
\end{proof}

To facilitate the construction of an example, we further look at the
functions $f_l(k) = 2^{l-1} m_l(k) + j_l(k)$ defined  above,
\eqref{e.defflk}-\eqref{e.flk2}. Assume that $k-1$ is written in the
binary representation:
$$k -1 = \sum_{j=1}^n k_j\cdot 2^{j-1} = \big( k_n k_{n-1}\dots k_2 k_1\big)_2. $$
By construction, $m_l(k)=(k_n k_{n-1} \dots k_{l+1})_2$, besides,
when $k_l=0$, we have $j_l(k)=0$, and if $k_l=1$,
$j_l(k)=(k_{l-1}\dots k_1)_2 +1$. Thus, $f_l(k)$ can be written in
closed form in terms of digits of $k-1$ as follows
\begin{equation}\label{e.flk}
f_l(k) = \sum_{j=l+1}^n k_j 2^{j-2} + k_l \cdot \sum_{j=1}^{l-1}k_j
2^{j-1} \, + k_l.
\end{equation}
Indeed, if $k_l=0$, the last two terms will disappear, otherwise,
they'll equal exactly $j_l(k)$.

Plugging this into \eqref{e.fk}, we obtain
\begin{equation}\label{e.kj}
 \sum_{p\in \mathcal V_n:\, p_1 < k/2^{n}} p_2  =
\sum_{l=1}^{n-1} \sum_{j=l+1}^n k_j \cdot 2^{j-l-2} + \sum_{l=1}^n
k_l \cdot 2^{-l} + \sum_{l=2}^n \sum_{j=1}^{l-1} k_j \cdot k_l \cdot
2^{j-l-1}.
\end{equation}
Obviously, the second term above is bounded by one. Next we shall
look at the first term in \eqref{e.kj}. At this point we assume that
\begin{equation}\label{e.halfdigits}
\sum_{j=1}^n k_j =\frac{n}{2} + O(1),
\end{equation}
 i.e. approximately half of the binary digits of $k-1$ are ones and half are
zeros. We have
\begin{align}
\nonumber \sum_{l=1}^{n-1} \sum_{j=l+1}^n k_j \cdot 2^{j-l-2} & =
\frac12 \sum_{j=2}^n k_j \cdot 2^{j-1} \sum_{l=1}^{j-1} 2^{-l}
\,\,\,\,\,\,\,\,\,\,\,\,\, = \,\,\, \frac12 \sum_{j=2}^n k_j \cdot
2^{j-1} \cdot (1-
2^{-(j-1)})\\
\label{e.kj1} & = \frac12 \sum_{j=2}^n k_j \cdot 2^{j-1}\, - \,
\frac12 \sum_{j=2}^n k_j \,\,\, = \,\,\, \frac12 k - \frac{n}{4} +
O(1).
\end{align}

As to the last term of \eqref{e.kj}, we have the following lemma:
\begin{lemma}\label{l.n8}
For every $n\in \mathbb N$, there exists $k: \, 1\le k \le 2^n$ with
   $\sum_{j=1}^n k_j =
n/2 + O(1)$, where $k -1  = \big( k_n k_{n-1}\dots k_2 k_1\big)_2$,
so that
\begin{equation}\label{e.kj3}
\sum_{l=2}^n \sum_{j=1}^{l-1} k_j \cdot k_l \cdot 2^{j-l-1} =
\frac{n}{8} + O(1).
\end{equation}
\end{lemma}

Assuming this statement for the moment,  putting together
\eqref{e.kj}, \eqref{e.kj1}, and \eqref{e.kj3} for $k$ defined by
Lemma \ref{l.n8} above, we obtain
\begin{equation}\label{e.1p2}
\sum_{p\in \mathcal V_n:\, p_1< k/2^{n}} p_2 =  \left( \frac12 k -
\frac{n}{4} \right) + \frac{n}{8} + O(1)  = \frac12 k - \frac{n}{8}
+ O(1),
\end{equation}
and together with \eqref{e.cn}, \eqref{e.I4}, this yields:
\begin{equation}\label{e.cnk}
\int_{[0,1]^2} C_{\mathcal V_n^\alpha} dx = \left( 2^{n-2}
+\frac{n}{8}\right) - \frac12 k + \left( \frac12 k - \frac{n}{8}
\right) + O(1) = 2^{n-2} + O(1).
\end{equation}
Finally, \eqref{e.cnk} and \eqref{e.I1} give
\begin{equation}\label{e.dnk}
\int_{[0,1]^2} D_{\mathcal V_n^\alpha}(x) dx = O(1).
\end{equation}

Thus, it remains to prove Lemma \ref{l.n8}. We shall denote
\begin{equation}\label{e.snk}
S(n,k-1) := \sum_{l=2}^n \sum_{j=1}^{l-1} k_j \cdot k_l \cdot
2^{j-l-1}
\end{equation}
and will look at some base examples first. Let $k'$ be of the form
\begin{equation}\label{e.k'}
k' := \big( 000111 \,\dots \, 000111 \big)_2 ,
\end{equation}
where the sequence $000111$ is repeated $n'$ times, $n=6n'$. We then
have the following calculation:
\begin{align}
\label{e.k'1} S(6n', k') & = \frac12 \sum_{l'=1}^{n'-1}
\left(2^{-(6l'+1)}+ 2^{-(6l'+2)}+2^{-(6l'+3)}\right) \cdot \left(
\sum_{j'=0}^{l'-1} \left[ 2^{6j'+1} +2^{6j'+2} +2^{6j'+3} \right]
\right)\\
\label{e.k'2} & \quad + \frac12 \sum_{l'=0}^{n'-1} \left(\frac12 +
\left( \frac12 +\frac14 \right) \right)\\
\nonumber & = \frac12 \sum_{l'=1}^{n'-1} 2^{-6l'} \cdot (2^{6l'}-1)
\cdot \frac{1}{2^6 - 1}\cdot (2^{-1}+2^{-2} + 2^{-3})(2^1 + 2^2 +
2^3
) \,\,+ \frac12 \cdot \frac54 n'\\
\nonumber & = \left(\frac{7}{72} +\frac{45}{72} \right)\cdot
\frac{n}{6} + O(1)\\
\label{e.k'3} & = \frac{13}{108}n + O(1),
\end{align}
where the term in \eqref{e.k'1} describes the interactions of digits
in different triples and \eqref{e.k'2} arises from interactions
within the triples. (Notice that the obtained fraction
$\frac{13}{108} \approx 0.12037...$ is quite close to the desired
$\frac18 = 0.125$.)

Next we set $k''=(00001111....00001111)_2$, where the string
$00001111$ is repeated $n''$ times. An absolutely analogous
computation yields:
\begin{align}
\nonumber S(8n'', k'') & = \frac12 \sum_{l''=1}^{n''-1} 2^{-8l''}
\cdot (2^{8l''}-1) \cdot \frac{1}{2^8 - 1}\cdot (2^{-1}+2^{-2} +
2^{-3}+2^{-4})(2^1 + 2^2 + 2^3 +2^4
)\\
\nonumber & \quad + \frac12 \sum_{l'=0}^{n'-1} \left(\frac12 +
\left( \frac12 +\frac14 \right) +\left( \frac12 +\frac14 +\frac18 \right)\right)\\
\label{e.k''3} & = \frac{19}{136}n + O(1).
\end{align}

We are now ready to define the number $k$ which satisfies
\eqref{e.kj3}. Set
\begin{equation}\label{e.k}
k -1 := \big( \underbrace{00001111 \,\dots \, 00001111}_{n_2 \,
\textup{digits}}\, \underbrace{000111 \,\dots \, 000111}_{n_1 \,
\textup{digits}}  \big)_2 .
\end{equation}
Then we have,
\begin{equation}
S(n_1+n_2,k-1)= S(n_1, k') + S(n_2, k'') + I(n_1, n_2),
\end{equation}
where $I(n_1, n_2)$ describes the interaction between the two parts
of $k$. We can estimate:
\begin{equation}
I(n_1,n_2) = \frac12 \cdot \left(\sum_{l=n_1+1}^n k_l 2^{-l}
\right)\cdot \left( \sum_{j=1}^{n_1} k_j 2^j \right) \leq \frac12
\big(2^{-n_1-1} \cdot 2 \big) \cdot (2^{n_1+1}-1) \leq 1.
\end{equation}
We now choose $n_1$ and $n_2$ so that
$\frac{n_1}{n_2}=\frac{54}{17}$, i.e. $n_1 = \frac{54}{71}n$,
$n_2=\frac{17}{71}n$. We then obtain
\begin{align}
\nonumber S(n, k-1)& = \frac{13}{108}n_1 + \frac{19}{136}n_2 +O(1)\\
\nonumber & = \left( \frac{13\cdot 54}{108 \cdot 71} + \frac{19
\cdot 17}{136 \cdot 71} \right) n + O(1)\\
& = \frac{n}{8} + O(1),
\end{align}
which finishes the proof of Lemma \ref{l.n8}. Thus, if we set
$\alpha_0 = 1 - \frac{k}{2^n}$, where $k$ is as defined in
\eqref{e.k}, then the cyclic shift of the Van der Corput set by
$\alpha_0$ satisfies
\begin{equation}\label{e.dnk0}
\int_{[0,1]^2} D_{\mathcal V_n^{\alpha_0}}\,(x)\, dx = O(1).
\end{equation}

{\em Remark. } Of course, the above construction only works when $n$
is a multiple of $71\cdot 2\cdot  4=568$. However, it can be easily
adjusted for other values of $n$ just by setting the ``remainder"
digits equal to zero.

\section{The Fourier coefficients of the discrepancy function}

Having eliminated the main problem, we shall now proceed to show
that the remaining part of $D_{\mathcal V_n}$ behaves well under
cyclic shifts. We shall use the exponential Fourier basis (rather
than the more standard in this theory Haar basis) since it is better
adapted to cyclic shifts.

Obviously, for any $\alpha$, we have
\begin{equation}\label{e.sumexp}
\sum_{p\in \mathcal V_n^\alpha} e^{-2\pi i m p_1} = \sum_{p\in
\mathcal V_n} e^{-2\pi i m p_1} = \sum_{j=0}^{2^n-1} e^{-2\pi i
\frac{m}{2^n}j} \cdot e^{-\pi i \frac{m}{2^{n}}} =
\begin{cases}
0,\,\,\,\,\,\,\,\,\,\textup{if}\,\,m\nequiv 0 \mod 2^n,\\
m,\,\,\,\,\,\,\,\,\textup{if}\,\, m=2^n m',\, m'\textup{- even},\\
-m,\,\,\,\,\textup{if}\,\, m=2^n m',\, m'\textup{- odd}.
\end{cases}
\end{equation}

\subsection*{Fourier coefficients in the case $n_1 , n_2 \neq 0$.}
We first note that, for $n_1 , n_2 \neq 0$, the Fourier coeficient
of the linear part is:
\begin{equation}\label{e.lnfourier}
\widehat{L_N}(n_1,n_2) = - \frac{N}{4\pi^2 n_1 n_2}.
\end{equation}
The counting part yields
\begin{equation}\label{e.cnfourier}
\widehat{C_{\mathcal V_N}}(n_1,n_2) = -\frac{1}{4\pi^2 n_1 n_2}
 \sum_{p\in \mathcal V_N} \left(1- e^{-2\pi i n_1 p_1} \right) \left(1- e^{-2\pi i n_2 p_2}
 \right),
\end{equation}
and, thus,
\begin{align}\label{e.dnfourier}
 \widehat{D_{\mathcal V_n}} (n_1, n_2) &= \frac{1}{4\pi^2 n_1 n_2}
 \sum_{p\in \mathcal V_N} \left( e^{-2\pi i n_1 p_1} + e^{-2\pi i n_2
 p_2} - e^{-2\pi i (n_1 p_1 +n_2 p_2)} \right)
\end{align}

We now consider cases:
\begin{itemize}
\item Both $n_1$ and $n_2 \equiv 0 \mod 2^n$. Then $\widehat{D_{\mathcal V_n}} (n_1,
n_2) = C \frac{N}{4\pi^2 n_1 n_2}$, where $C$ takes values $-3$ or
$1$, depending on whether $n_1/2^n$ and $n_2/2^n$ are even or odd.
\item $n_1 \nequiv 0 \mod 2^n$, $n_2 \equiv 0 \mod 2^n$. In this case $\widehat{D_{\mathcal V_n}} (n_1,
n_2) =  \frac{N}{4\pi^2 n_1 n_2}\cdot e^{-\pi i n_2/2^n}$.
\item $n_2 \nequiv 0 \mod 2^n$, $n_1 \equiv 0 \mod 2^n$. In this case $\widehat{D_{\mathcal V_n}} (n_1,
n_2) =  \frac{N}{4\pi^2 n_1 n_2}\cdot e^{-\pi i n_1/2^n}$.
\item $n_1, n_2 \nequiv 0 \mod 2^n.$ Now we have $\widehat{D_{\mathcal V_n}} (n_1,
n_2) = -\frac{1}{4\pi^2 n_1 n_2}
 \sum_{p\in \mathcal V_N}  e^{-2\pi i (n_1 p_1 +n_2
 p_2)}$.
 \end{itemize}
 Changing $p_1$ to $(p_1 +\alpha) \mod 1$ in the above computations,
 with $\alpha = j/2^n$, we  notice that
 \begin{equation}\label{e.nochange1}
 \Abs{\widehat{D_{\mathcal V_n^\alpha}}(n_1,n_2)} =  \Abs{\widehat{D_{\mathcal
 V_n}}(n_1,n_2)} \,\,\,\,\textup{when}\,\, n_1, n_2 \neq 0.
 \end{equation}
 Indeed, in the first three cases the coefficient does not change,
 while in the last it is multiplied by $e^{-2\pi i n_1 \alpha}$.

 \subsection*{Fourier coefficients in the case $n_2=0$, $n_1\neq 0$.}
 We first note that, in this case
 \begin{equation}\label{e.lnfourier0}
\widehat{L_N}(n_1, 0) = -\frac{N}{4\pi i n_1},
 \qquad\textup{and}\qquad
\widehat{C_{\mathcal V_n}}(n_1, 0) = -\frac{1}{2 \pi i n_1}
 \sum_{p\in \mathcal V_n}  \left(1- e^{-2\pi i n_1 p_1}\right) \left(1- p_2 \right)
 .
 \end{equation}
 Thus, taking into account \eqref{e.sumall}, we have
 \begin{equation}\label{e.dnfourier0}
 \widehat{D_{\mathcal V_n}} (n_1,0) = \frac{1}{2 \pi i n_1} \sum_{p\in \mathcal V_n}  e^{-2\pi i
 n_1
 p_1}\cdot \left(1- p_2 \right).
 \end{equation}
 And, once again, we obtain that
 \begin{equation}\label{e.nochange2}
{\widehat{D_{\mathcal V_n^\alpha}}(n_1,0)} =  \widehat{D_{\mathcal
 V_n}}(n_1,0) \cdot e^{-2\pi i n_1 \alpha},
 \qquad\textup{i.e.}\,\,\Abs{\widehat{D_{\mathcal V_n^\alpha}}} =  \Abs{\widehat{D_{\mathcal
 V_n}}} \,\,\textup{if }\, n_1\neq 0, n_2= 0.
 \end{equation}

\subsection*{Fourier coefficients in the case $n_1=0$, $n_2\neq
0$.}As above, we can compute
 \begin{equation}\label{e.dnfourier00}
 \widehat{D_{\mathcal V_n}} (0, n_2) = \frac{1}{2 \pi i n_2} \sum_{p\in \mathcal V_n} \left(1- p_1 \right)\cdot e^{-2\pi i n_2
 p_2}.
 \end{equation}
 In the case $n_2 \equiv 0 \mod 2^n$, we obtain, using
 \eqref{e.sumall},
 \begin{equation}\label{e.nochange3}
 \widehat{D_{\mathcal V_n}} (0,
 n_2) = \widehat{D_{\mathcal V_n^\alpha}} (0, n_2)= \frac{N}{4 \pi i
 n_2}\cdot e^{-\pi i n_2/2^n}.
 \end{equation}

 The only somewhat non-trivial case is when $n_1 = 0$, $n_2 \nequiv 0
 \mod 2^n$.
 The Fourier coefficient in this case is
\begin{align}\label{e.dnfourier01}
 \nonumber \widehat{D_{\mathcal V_n^\alpha}} (0, n_2) & = \frac{1}{2 \pi i n_2} \sum_{p\in \mathcal
 V_n^\alpha} \left(1- p_1 \right)\cdot e^{-2\pi i n_2
 p_2}\\
 & = \widehat{D_{\mathcal V_n}} (0, n_2) + \frac{1}{2 \pi i n_2} \sum_{p\in \mathcal V_n :\, p_1>k/2^n}
   e^{-2\pi i n_2
 p_2},\,\,\,\,\,\textup{where}\,\,k/2^n=1-\alpha.
 \end{align}
We shall examine the last sum above. Assume $n_2 = 2^s m$, where
$0\le s <n$, $m$ is odd. Let us look over the part of the sum,
ranging over a dyadic interval of length $2^{-l}$, $1\le l \le n$.
This means that the first $l$ digits of $p_1$ (and thus, the last
$l$ digits of $p_2$) are fixed, and the last $n-l$ (the first $n-l$
of $p_2$) are allowed to change freely.
\begin{align}\label{e.dyadick}
\sum_{p\in \mathcal V_n :\, p_1 \in [q2^{-l}, (q+1)2^{-l})}
   e^{-2\pi i n_2  p_2}
& = e^{-2\pi i  2^s m \left( q_{n-l+1} 2^{-n+l-1} +\dots + q_{n}
2^{-n} + 2^{-n-1} \right)} \cdot \sum_{j=0}^{2^{n-l}-1} e^{-2\pi i m
2^{-n+l+s}j}.
\end{align}
It is easy to see that the last sum equals zero when $l+s<n$;
otherwise, its absolute value is at most $2^{n-l}$. 
We now split the interval $\{p_1 > k/2^n\}$ into at most $n$  dyadic
intervals of length $2^{-l}$,  $1\le l \le n$. We obtain
\begin{equation}\label{e.diff}
\ABS{\sum_{p\in \mathcal V_n :\, p_1>k/2^n}
  e^{-2\pi i n_2
 p_2}} \le \sum_{l=n-s}^n 2^{n-l} = 2^{s+1}-1.
\end{equation}
That is, for $n_2=2^s m$, by \eqref{e.dnfourier01} and
\eqref{e.diff}, we have
\begin{equation}\label{e.diff1}
\ABS{\widehat{D_{\mathcal V_n^\alpha}} (0, n_2)-\widehat{D_{\mathcal
V_n}} (0, n_2)} \le \frac{2^{s+1}}{2\pi n_2} = \frac{1}{\pi m}.
\end{equation}

\section{Proof of Theorem \ref{t.main}.}

For a function $f\in L^2\left([0,1]^2\right)$ and $S\subset \mathbb
Z^2$, we shall denote by $f_S$ the orthogonal projection of $f$ onto
the span of the Fourier terms with indices in $S$, i.e.
\begin{equation}\label{e.fs}
f_S (x_1,x_2) \eqdef \sum_{(n_1, n_2 )\in S} \widehat{f} (n_1,n_2)\,
e^{2\pi i (n_1 x_1 + n_2 x_2)}.
\end{equation}

Due to \eqref{e.nochange1}, \eqref{e.nochange2}, and Parseval's
identity, we have
\begin{equation}
\NOrm \big( D_{\mathcal V_n^{\alpha_0}}\,\big)_{\mathbb Z^2
\setminus \{n_1=0\}} .2. = \NOrm \big( D_{\mathcal
V_n}\,\big)_{\mathbb Z^2 \setminus \{n_1=0\}} .2.
\end{equation}
Inequality \eqref{e.diff1} yields
\begin{equation}
\NOrm \left( D_{\mathcal V_n^{\alpha_0}} - D_{\mathcal
V_n}\right)_{\{n_1=0, n_2\neq 0\}}.2.^2 \lesssim \sum_{s=0}^{n-1}
\sum_{m\,\,\textup{odd}} \frac1{m^2} \lesssim n= \log N.
\end{equation}
Thus, we see that $\Norm (D_{\mathcal V_n})_{\mathbb Z^2\setminus
(0,0)}.2.$ indeed  does not change much under cyclic shift. The
inequalities above and \eqref{e.rest} yield:
\begin{equation}
\Norm \big(D_{\mathcal V_n^{\alpha_0}}\big)_{\mathbb Z^2\setminus
(0,0)}.2. \lesssim \Norm \big(D_{\mathcal V_n}\big)_{\mathbb
Z^2\setminus (0,0)}.2. + \left( \log N \right)^{1/2} \lesssim \left(
\log N \right)^{1/2}.
\end{equation}
Together with the fact that $\int D_{\mathcal V_n^{\alpha_0}}
\lesssim 1$, \eqref{e.dnk0}, this finishes the proof:
\begin{equation}
\Norm D_{\mathcal V_n^{\alpha_0}} .2. \lesssim \left( \log N
\right)^{1/2}.
\end{equation}









 \end{document}